%% file: halves_v2.tex
\documentclass[12pt]{article}
\usepackage{razb}

\input{razb}

\usepackage{amsfonts}
\usepackage{amsmath}
\usepackage{mathabx}
\usepackage{epic,eepic}
\usepackage{stmaryrd}
\title{More about sparse halves in triangle-free graphs}
\author{Alexander Razborov\thanks{University of Chicago, USA, {\tt razborov@math.uchicago.edu}
and Steklov Mathematical Institute, Moscow, Russia, {\tt razborov@mi.ras.ru}.}}
\begin{document}
\maketitle

\begin{abstract}
One of Erd\H{o}s's conjectures states that every triangle-free graph on $n$
vertices has an induced subgraph on $n/2$ vertices with at most $n^2/50$
edges. We report several partial results towards this conjecture. In
particular, we establish the new bound $\frac{27}{1024}n^2$ on the number
of edges in general case. We completely prove the conjecture for graphs of
girth $\geq 5$, for graphs with independence number $\geq 2n/5$ and for
strongly regular graphs. Each of these three classes includes both known
(conjectured) extremal configurations, the 5-cycle and the Petersen graph.
\end{abstract}

\section{Introduction} \label{sec:intro}
Throughout his long career, Erd\H{o}s repeatedly \cite{Erd6,Erd3,Erd7} asked
several questions united by one common theme: how far from being bipartite
can a triangle-free graph be. One of them, the ``pentagon problem'', was
completely solved in \cite{erdos,Grz}. Another question asks what can be the
maximum possible $\ell_1$-distance (which in this case is simply the number
of edges deleted) from a triangle-free graph to the class of bipartite
graphs. It was studied in \cite{EFP*,EGS,BCL}.

This paper is devoted to the third question, ``half-graph'' conjecture
sometimes referred to as ``one of Erd\H{o}s's favorite'' \cite{KeS2}. Given a
triangle-free graph $G$, is it always possible to remove half of its vertices
such that the edge density $\frac{|E(G)|}{2|V(G)|^2}$ becomes $\leq 1/25$? In
this direction, there has been more recent work done \cite{EFR*, Kri,
KeS2,NoY} although the conjecture still remains widely open.

\medskip
In this paper we improve on several statements from those papers and offer
some new results.

Fix a triangle-free graph $G$ on $n$ vertices and let $\beta(G)$ be the
minimum number of edges in its half-graphs, normalized\footnote{It would have
been much more natural to normalize by $\frac{n^2}2$ instead but we prefer
our notation to be consistent with the literature.} by $n^2$. The {\em
half-graph conjecture} by  Erd\H{o}s says that $\beta(G)\leq \frac 1{50}$,
for any triangle-free $G$. The bound $\beta(G)\leq \frac 1{16}$ is obvious
(attained by the random half), \cite{EFR*} proved that $\beta(G)\leq \frac
1{30}$ and \cite{Kri} improved this to $\beta(G)\leq \frac 1{36}$.

\smallskip \noindent
{\bf Theorem.} {\em For any triangle-free graph $G$, $\beta(G)\leq
\frac{27}{1024}$.}

The number $\frac{27}{1024}$ here is not arbitrary, it reflects what can be
achieved with a certain class of methods, and the Clebsch graph is an
extremal example for the resulting extremal problem. We will comment more on
it below.

\smallskip
The (conjectural) extremal examples in the half-graph conjecture are the
pentagon $C_5$ and the Petersen graph. The former does not contain induced
matching of size 2 as well.

\smallskip \noindent
{\bf Theorem.} {\em The half-graph conjecture is true for any {\rm
(}triangle-free{\rm )} graph without induced matchings of size 2}.

\smallskip
Before going any further, let us briefly discuss the proofs of these two
theorems as they bring about potentially interesting concepts and questions.

\medskip \noindent
{\bf Digression on quadriliterals counting.} Let $\rho=\rho(G)$ and
$C_4=C_4(G)$ be the edge density of $G$ and the density of quadriliterals
(copies of $C_4$) in it. They are computed in the sense of flag
algebras/graph limits: $G$ is replaced first with its infinite blow-up (so
that in particular copies of the path $P_3$ in $G$ and even individual edges
contribute to $C_4(G)$). These two quantities are of fundamental importance
in the theory of quasi-random graphs: $C_4\geq 3\rho^4$, and an increasing
sequence of graphs with the same value of $\rho(G)$ is quasi-random if and
only if this inequality is asymptotically tight \cite{CGW}.

For triangle-free graphs this bound can be easily improved to
\begin{equation} \label{eq:trivial_bound}
C_4 \geq \frac{3\rho^4}{1-\rho}.
\end{equation}
This is a quantitative refinement of the statement that triangle-free graphs
are not quasi-random, and thus it is natural to ask: what is the {\em
smallest} value of $C_4(G)$ as a function of $\rho(G)$? In a sense, it is a
dual to Erd\H{o}s's questions. The latter ask, in one or another form, how
far from being {\em bi-partite} a triangle-free graph {\em can} be. The
``quadriliteral question'', on the contrary, is asking how far from {\em
quasi-random} a triangle-free graph {\em must} be.

We expect this question to be extremely difficult in general. But it is very
tightly related to Erd\H{o}s's conjectures as was already demonstrated in
\cite[Section 2]{EFP*}. In our context, an easy analysis of [part of]
Krivelevich's proof, followed by a straightforward averaging gives:

\begin{proposition} \label{prop:krivilevich}
$$
\beta(G)\leq \frac 18\rho(G) - \frac{C_4(G)}{12\rho(G)}.
$$
\end{proposition}

Both bounds on $\beta(G)$ then follow from the following

\smallskip \noindent
{\bf Theorem.} {\em
\begin{enumerate}
\item For any triangle-free graph $G$,
$$
C_4(G) \geq \frac 32\rho(G)^2 - \frac{81}{256} \rho(G).
$$

\item For any triangle-free graph $G$ without induced matchings of size 2,
$$
C_4(G) \geq \frac 32\rho(G)^2 - \frac{6}{25} \rho(G).
$$
\end{enumerate}
}

This theorem is proved via a ``medium size'' flag-algebraic calculation. The
second bound is tight for $\rho=2/5$, as it must be since the pentagon $C_5$
is the (conjectural) extremal example for the half-graph conjecture. The
first inequality beats the trivial bound \eqref{eq:trivial_bound} for
$0.257\leq\rho \leq 0.366$. It is tight for $\rho=5/16$, the extremal example
being the Clebsch graph, and this seems to be the only non-trivial value of
$\rho(G)$ for which we know the exact solution to the quadriliteral problem.

\smallskip
We further remark that our bound $\beta(G)\leq \frac{27}{1024}$ is tight for
a reasonably natural restriction of Erd\H{o}s's conjecture. More
specifically, many previous results, including Proposition
\ref{prop:krivilevich}, are based on the following simple construction. Pick
an edge $e\in E(G)$. It naturally defines a splitting of $V(G)$ into three
parts. The sparse half constructed by this method is determined by assigning
the same weight within each of these parts. Then the value $\frac{27}{1024}$
in the half-graph conjecture is the best one can achieve using this method,
with the Clebsch graph being (again) an extremal example.

\medskip
Let us now return to reviewing the remaining results. The half-graph
conjecture is obvious if $\rho(G)\leq 4/25$ (once more, the random half will
do the job). Keevash and Sudakov \cite{KeS2} relaxed this to $\rho(G)\leq
1/16$.

\smallskip \noindent
{\bf Theorem.} {\em The half-graph conjecture is true for any triangle-free
graph with $\rho(G)\leq \frac{33-\sqrt{161}}{116}\approx 0.1751$.}

\smallskip
Based on this theorem, we prove the following:

\smallskip \noindent
{\bf Theorem.} {\em The half-graph conjecture is true for any triangle-free
strongly regular graph.}

\smallskip
A significant amount of activity took place around the critical value
$\rho=2/5$. \cite[Theorem 3]{Kri} proved the conjecture for regular
triangle-free graphs with $\rho(G)\geq 2/5$ and \cite{KeS2} removed the
restriction of regularity. Norin and Yepremyan \cite{NoY} improved this
result by relaxing the assumption $\rho(G)\geq 2/5$ to $\rho(G)\geq
2/5-\gamma$, where $\gamma>0$ is a (calculable) constant. When $\rho(G)$ is
replaced by the (normalized) {\em minimum} degree $\delta(G)$, the bound on
$\gamma$ significantly improves and the half-graph conjecture is true
whenever $\delta(G)\geq \frac 5{14}$ \cite{NoY}.

\smallskip \noindent
{\bf Theorem.} {\em Let $\alpha(G)$ be the normalized  {\rm (}by $n${\rm )}
independence number of $G$, and assume that $\alpha(G)\geq 3/8$. Then
$$
\beta(G)\leq \frac 12\alpha(G)\of{\frac 12-\alpha(G)}.
$$
}

\noindent
{\bf Corollary.} {\em The half-graph conjecture holds for any
triangle-free graph with {\rm (}normalized{\rm )} {\bf maximum} degree $\geq
2/5$.}

 Note that unlike the previous results we do not require {\em all}
 vertices to have large degree, even on average, but just one. Also, this
 theorem covers the Petersen graph as well since it has (unnormalized) independence
 number 4.  On the negative side, we have not been able to extend it to an
 open neighbourhood of $2/5$ as the previous work did.

\smallskip
 Finally, both conjectured extremal examples have girth 5.

 \smallskip \noindent
 {\bf Theorem.} {\em The half-graph conjecture holds for all graphs of girth
 $\geq 5$.}

 \medskip
 The rest of the paper is organized as follows. In Section \ref{sec:prel} we
 give all necessary definitions. In Section \ref{sec:results} we re-state
 our results, mostly as a matter of convenience. Section \ref{sec:proofs} is
 devoted to proofs, and we conclude in Section \ref{sec:concl} with a few
 remarks and open questions.

 \section{Preliminaries} \label{sec:prel}

 Unless specified otherwise, all graphs $G$ in this paper are finite, simple and
 triangle-free. $N_G(v)$ is the neighbourhood of $v$, and $n$ will always stand
 for the number of vertices. For disjoint sets of vertices $X$ and $Y$,
 $E(X,Y)$ is the set of cross-edges between $X$ and $Y$; likewise, $E(X)$ is
 the set of edges induced by $X$.

 A {\em half} in a graph $G$ with $n$ vertices is a function
 $\mu\function{V(G)}{[0,1]}$ such that $\sum_{v\in V(G)}\mu(v)=n/2$. We let
 $$
\beta(G,\mu) \df \frac 1{n^2}\sum_{(u,v)\in E(G)}\mu(u)\mu(v)
 $$
and
$$
\beta(G) \df \min\vecset{\beta(G,\mu)}{\mu\ \text{is a half}}.
$$
It is easy to see by a simple convexity argument that when $n$ is even, the
minimum is attained at a 0-1 half in which case we will also use the notation
$\beta(G,A),\ A \in {V(G)\choose n/2}$. When $n$ is odd, it is attained at an
almost 0-1 half $\mu$, that is $\mu(v_0)=1/2$ for a single vertex $v_0$ and
$\mu(v)\in \{0,1\}$ for all others. But the analytical form above is
extremely handy in concrete constructions, as we shall see.

\begin{conjecture}[Half-graph conjecture by Erd\H{o}s] \label{conj:erdos}
$\beta(G)\leq \frac 1{50}$ for any triangle-free graph $G$.
\end{conjecture}

Let $\rho(G)$ and $C_4(G)$ be the edge density and the density of
quadriliterals {\em defined consistently with flag algebras}. That is,
$$
\rho(G)\df \frac{2|E(G)|}{n^2}
$$
and in order to compute $C_4(G)$ we sample $\rn{v_i}\in V(G)\ (i\in [4])$
uniformly and completely independently (that is, with repetitions), form the
graph on $[4]$ with the set of edges $\set{(i,j)\in {[4]\choose
2}}{(\rn{v_i}, \rn{v_j})\in E(G)}$ and let $C_4(G)$ be the probability that
it is {\em isomorphic} to $C_4$.

For $v\in V(G)$ we let
$$
e(v) \df \frac{|N_G(v)|}{n}
$$
be the relative degree of $v$ and
$$
\Delta(G)\df\max\vecset{e(v)}{v\in V(G)}
$$
be the maximum degree, also relative. Likewise,
$$
\alpha(G) \df \max\vecset{\frac{|A|}{n}}{A\ \text{an independent set}}
$$
is the relative independence number.

We can assume w.lo.g. that $\alpha(G)\leq 1/2$ since otherwise Conjecture
\ref{conj:erdos} is obvious. Thus,
\begin{equation} \label{eq:assumption}
\rho(G) \leq \Delta(G) \leq \alpha(G) \leq 1/2.
\end{equation}

For sets of vertices $A,B,C,D,E,\ldots$ we will denote by
$p_a,p_b,p_c,\ldots$ their densities:
$$
p_x \df \frac {|X|}{n}.
$$
Let $\rho_{xy}\ (x\neq y\in \{a,b,c,d,\ldots\},\ X\cap Y=\emptyset)$ be the
normalized density of cross-edges:
$$
\rho_{xy} \df \frac{2|E(X,Y)|}{n^2},
$$
and, likewise,
$$
\rho_{x} \df \frac{|E(X)|}{n^2}.
$$
Finally, for $v\not\in X$, let
$$
e_X(v) \df \frac{|N_G(v)\cap X|}{n}.
$$

\section{Results} \label{sec:results}
In this section we collect in one place our main results stated in the
introduction.

\begin{theorem} \label{thm:quadriliterals}
\begin{tenumerate}
\item  For any triangle-free graph $G$,
$$
C_4(G) \geq \frac 32\rho(G)^2 - \frac{81}{256} \rho(G)
$$
{\rm (}the bound is tight for the Clebsch graph{\rm ).}

\item For any triangle-free graph $G$ without induced matchings of size 2,
$$
C_4(G) \geq \frac 32\rho(G)^2 - \frac{6}{25} \rho(G).
$$
{\rm (}the bound is tight for $C_5${\rm ).}
\end{tenumerate}
\end{theorem}

\begin{theorem} \label{thm:improved_bounds}
For any triangle-free graph $G$, $\beta(G)\leq \frac{27}{1024}$.
\end{theorem}

\begin{theorem} \label{thm:no_matchings}
Conjecture {\rm \ref{conj:erdos}} is true for any triangle-free graph without
induced matchings of size 2.
\end{theorem}

\begin{theorem} \label{thm:small_density}
Conjecture {\rm \ref{conj:erdos}} is true for any triangle-free graph with
$\rho(G)\leq \rho_0\df \frac{33-\sqrt{161}}{116}$.
\end{theorem}

Recall that a regular triangle-free graph $G$ is {\em strongly regular} if
$|N_G(v)\cap N_G(w)|$ takes the same value $c$ for all pairs $(v,w)$ of
non-adjacent vertices.

\begin{theorem} \label{thm:strongly_regular}
Conjecture {\rm \ref{conj:erdos}} is true for any triangle-free strongly
regular graph.
\end{theorem}

\begin{theorem} \label{thm:large_alpha}
For any triangle-free graph $G$ with $\alpha(G)\geq 3/8$ we have
$$
\beta(G)\leq \frac 12\alpha(G)\of{\frac 12-\alpha(G)}.
$$
\end{theorem}

\begin{corollary}
Conjecture {\rm \ref{conj:erdos}} is true for any triangle-free graph with
$\alpha(G)\geq 2/5$.
\end{corollary}

\begin{theorem} \label{thm:girth}
Conjecture {\rm \ref{conj:erdos}} is true for any triangle-free graph of
girth $\geq 5$.
\end{theorem}

\section{Proofs} \label{sec:proofs}

In this section we prove all our results. Some of the proofs, particularly in
Sections \ref{sec:flag} and \ref{sec:sparse}, heavily rely on symbolic Maple
computations. The corresponding worksheet, along with some supporting
material, can be found at\\
http://people.cs.uchicago.edu/\~{}razborov/files/halves.zip.

\subsection{Flag-algebraic calculations} \label{sec:flag}

In this section we prove Theorem \ref{thm:quadriliterals}. As we remarked in
Section \ref{sec:prel}, our notation for finite graphs is consistent with
flag algebras hence it is sufficient to prove the inequalities
\begin{eqnarray}
\label {eq:first_inequality}  \frac 32\rho^2 - \frac{81}{256}\rho &\leq& C_4\\
\label{eq:second_inequailty} \frac 32\rho^2 - \frac{6}{25}\rho &\leq& C_4 +2M_4
\end{eqnarray}
($M_4$ is the matching with two edges) in the theory $T_{\text{TF}}$ of
triangle-free graphs and then apply them to the infinite (balanced) blow-up
of $G$.

We do it by a straightforward Cauchy-Schwartz computation in flag algebras.
Since quite a number of those have already appeared in the literature, with
varying degree of informal explanation, we do ours matter-of-factly strictly
adhering to the notation of \cite{flag}.

\medskip
Let us start with \eqref{eq:first_inequality}; for that we need to consider
triangle-free graphs on 8 vertices. We have $\absvalue{\mathcal M_8}=410$ and
$\absvalue{\mathcal F_6^{\sigma_i}}=d_i$, where $d_1=110,\ d_2=81,\ d_3=67,\
d_4= 46$ and the types $\sigma_i$ are shown on Figure \ref{fig:types} (with
the exception of $\sigma_4$, these are the same types employed in
\cite{wheel}).
\begin{figure}[htb]
\begin{center}
\input{types.eepic}
\caption{\label{fig:types} Types.}
\end{center}
\end{figure}
We enumerate flags in $\mathcal F_6^{\sigma_i}$ in a rather arbitrary order
as $\mathcal F_6^{\sigma_i} = \{F_1^{\sigma_i},\ldots, F_{d_i}^{\sigma_i}\}$
and exhibit PSD matrices $Q_i$ of size $d_i\times d_i$ with rational
coefficients such that
\begin{equation} \label{eq:clebsch}
\sum_{i=1}^4 \sum_{j_1,j_2=1}^{d_i} \eval{Q_i(j_1,j_2)F_{j_1}^{\sigma_i}
F_{j_2}^{\sigma_i}}{\sigma_i} \ll_8 C_4-\frac 32\rho^2+\frac{81}{256}\rho,
\end{equation}
where $\ll_8$ means coefficient-wise comparison after expressing both sides
of this inequality as linear combinations of the elements of $\mathcal M_8$.

The only further remark we want to make here is that the matrices $Q_i$ are
degenerate and their co-ranks $d_i-\text{rk}(Q_i)$ are equal to 2,2,5,4,
respectively. This reflects the fact (and makes an excellent sanity check for
our calculations) that the Clebsch graph $G_{\text{Clebsch}}$ is an extremal
configuration for the inequality \eqref{eq:clebsch}. Hence every strict
homomorphism $\sigma_i\to G_{\text{Clebsch}}$ gives rise to an element in the
kernel of $Q_i$. The actual computation is deferred to
http://people.cs.uchicago.edu/\~{}razborov/files/halves.zip.

\smallskip
The inequality \eqref{eq:second_inequailty} is proved similarly, but this
time we need only graphs on 6 vertices; on the other hand, instead of
$\sigma_3$ we need the type $E$. We have $\absvalue{\mathcal M_6}=38$,
$|\mathcal F_6^{\sigma_i}|=d_i$, where $d_1=12,\ d_2=10,\ d_4 =7$, and also
$|\mathcal M_4^E|=10$. The computation has the form
$$
\sum_{j_1,j_2=1}^{10}\eval{R_E(j_1,j_2)F_{j_1}^{E}
F_{j_2}^{E}}{E} + \sum_{i\in \{1,2,4\}} \sum_{j_1,j_2=1}^{d_i} \eval{R_i(j_1,j_2)F_{j_1}^{\sigma_i}
F_{j_2}^{\sigma_i}}{\sigma_i} \ll_6 C_4+2M_4-\frac 32\rho^2+\frac{6}{25}\rho.
$$

The coefficient 2 in front of $M_4$ is rather arbitrary, we did no attempt to
optimize on it. As this inequality is tight on $C_5$, matrices
$R_E,R_1,R_2,R_4$ also must be degenerate and indeed they have co-ranks
1,1,1,3, respectively.

\subsection{Absolute lower bounds on $\beta(G)$} \label{sec:absolute}

In this section we establish Theorems \ref{thm:improved_bounds} and
\ref{thm:no_matchings}. As was already mentioned, they immediately follow
from Theorem \ref{thm:quadriliterals} and Proposition \ref{prop:krivilevich}
so it only remains to prove the latter. This is simply a part of
Krivilevich's argument \cite{Kri}, slightly re-phrased, but we include it
here for the sake of completeness.

\smallskip
Let us start with considering an individual edge $(v_1,v_2)\in E(G)$. Denote
$A_i\df N_G(v_i)$, and let $e_i\df e(v_i) (=p_{a_i})$; recall that $e_i\leq
1/2$ by \eqref{eq:assumption}. Let
$$
p_i\df \frac{1/2-e_i}{1-e_1-e_2}
$$
so that $p_1+p_2=1$, and let $B\df V(G)\setminus (A_1\cup A_2)$. For $i=1,2$
define the half $\mu_i$ by
$$
\mu_i(v)\df \begin{cases}
1, & \text{if}\ v\in A_i\\
p_i, & \text{if}\ v\in B\\
0, & \text{if}\ v\in A_{3-i}.
\end{cases}
$$
Then
\begin{eqnarray}
\label{eq:mu_1} 2\beta(G) &\leq& 2\beta(G,\mu_1) =p_1\rho_{a_1,b} +p_1^2\rho_b,\\
\label{eq:mu_2} 2\beta(G) &\leq& 2\beta(G,\mu_2) =p_2\rho_{a_2,b} +p_2^2\rho_b.
\end{eqnarray}

Multiplying the $i$th inequality here by $p_{3-i}$ and adding them together,
we get
$$
2\beta(G) \leq p_1p_2(\rho_{a_1b} + \rho_{a_2b} +\rho_{b}) = p_1p_2(\rho-\rho_{a_1a_2})
\leq \frac 14(\rho- C_4^E(v_1,v_2)),
$$
where we denoted $\rho_{a_1a_2}$ by $C_4^E(v_1,v_2)$ to stress that this is
the contribution of $(v_1,v_2)$ to $C_4(G)$. Finally, averaging this over all
edges, we get
$$
2\rho\beta(G) \leq \frac 14\eval{\rho-C^4_E}{E} = \frac 14(\rho^2 - \frac 23C_4)
$$
that is precisely Proposition \ref{prop:krivilevich}.

\subsection{Sparse graphs} \label{sec:sparse}

In this section we prove Theorem \ref{thm:small_density}. As in the previous
work \cite{KeS2}, the analysis splits into two cases: $\Delta(G)\geq 1/4$ and
$\Delta(G)\leq 1/4$.

The first case is taken care of by the following variant of Proposition
\ref{prop:krivilevich}:

\begin{lemma} \label{lem:density}
$$
\beta(G) \leq \frac{\rho(G)(1-2\Delta(G))}{8(1-\Delta(G))^2}.
$$
\end{lemma}

\begin{proof}
Pick $v\in V(G)$ with $e(v) = \Delta\ (\df \Delta(G))$, and let $A\df N_G(v)$
(so that $p_a=\Delta$) and $B=N_G(v)\setminus A$. Construct the following
halves $\mu_0$ and $\mu_1$:
\begin{eqnarray*}
\mu_0(w) &\df& \begin{cases} 0,\ \text{if}\ w\in A\\ \frac 1{2(1-\Delta)},\
\text{if}\ w\in B  \end{cases}\\
\mu_1(w) &\df& \begin{cases} 1,\ \text{if}\ w\in A\\ \frac{1/2-\Delta}{1-\Delta},\
\text{if}\ w\in B  \end{cases}
\end{eqnarray*}

Then
\begin{eqnarray}
\label{eq:s0} 2\beta(G) \leq 2\beta(G,\mu_0) &=& \frac{\rho_b}{4(1-\Delta)^2}\\
\label{eq:s1} 2\beta(G) \leq 2\beta(G,\mu_1) &=& \frac{1/2-\Delta}{1-\Delta}\rho_{ab}
+ \of{\frac{1/2-\Delta}{1-\Delta}}^2\rho_b.
\end{eqnarray}
Multiplying \eqref{eq:s1} by $1-2\Delta$ and adding it to \eqref{eq:s0}, we
get
$$
4(1-\Delta)\beta(G) \leq \frac{1-2\Delta}{2(1-\Delta)}(\rho_{ab}+\rho_b) =
\frac{1-2\Delta}{2(1-\Delta)}\rho(G).
$$
\end{proof}

Now, the function $\frac{(1-2\Delta)}{8(1-\Delta^2)}$ is decreasing for
$\Delta\in [1/4,1/2]$, hence $\Delta(G)\geq 1/4$ implies $\beta(G)\leq
\frac{\rho(G)}{9}$ and then Theorem \ref{thm:small_density} follows since
$\rho_0\leq \frac 9{50}$.

\smallskip
The case $\Delta(G)\leq 1/4$ is more difficult. As in the proof of
Proposition \ref{prop:krivilevich}, let us first consider an individual edge
$(v_1,v_2)\in E(G)$ (but this time we will not randomize over this choice but
will pick it up in a way to be specified later). We will re-use the notation
$e_i,A_i,B,p_i$ from that proof so that we still have the bounds
\eqref{eq:mu_1}, \eqref{eq:mu_2}. But now the condition $\Delta(G)\leq 1/4$
allows us to form one more half
$$
\mu_0(G) \df \begin{cases} 1,\ \text{if}\ v\in A_1\cup A_2\\ q, \text{if}\
v\in B,
\end{cases}
$$
where
$$
p_0\df \frac{1/2-e_1-e_2}{1-e_1-e_2}.
$$
This leads to the extra bound
\begin{equation} \label{eq:mu0}
2\beta(G) \leq \rho_{a_1a_2} + p_0(\rho_{a_1b} + \rho_{a_2b}) + p_0^2\rho_b.
\end{equation}

We are now looking for non-negative coefficients $\alpha_0,\alpha_1,\alpha_2$
such that multiplying by them \eqref{eq:mu0}, \eqref{eq:mu_1} and
\eqref{eq:mu_2}, respectively, and adding up the results, we will equalize
the coefficients in front of $\rho_{a_1b}, \rho_{a_1a_2}$, as well as
$\rho_{a_2b}, \rho_b$. For that purpose we set
\begin{eqnarray*}
\alpha_0 &\df& (1-2e_1)^2(1-2e_2)\\
\alpha_1 &\df& (1-2e_1)(1-2e_2)\\
\alpha_2 &\df& 2(1-2e_1)e_2.
\end{eqnarray*}
Then (see the Maple worksheet)
\begin{eqnarray*}
&&4(1-2e_1)(1-e_1-e_2+2e_1e_2)\beta(G)\leq \alpha_0(\rho_{a_1a_2}+\rho_{a_1b}) +
\gamma(\rho_{a_2b}+\rho_b)\\ && \hspace{\parindent}  = (\alpha_0-\gamma)(\rho_{a_1a_2}+\rho_{a_1b}) +
\gamma(\rho_{a_1a_2}+\rho_{a_1b} +\rho_{a_2b}+\rho_b )\\ && \hspace{\parindent} = (\alpha_0-\gamma)(\rho_{a_1a_2}+\rho_{a_1b})
+ \gamma\rho,
\end{eqnarray*}
where
$$
\gamma\df \frac{(1-2e_1)(1-2e_2)(1-4e_1+4e_1^2+4e_1e_2)}{2(1-e_1-e_2)}.
$$
Note for the record that
$$
\alpha_0-\gamma = \frac{(1-2e_1)(1-2e_2)(1-2e_1-2e_2)}{2(1-e_1-e_2)} \geq 0
$$
since $e_1,e_2\leq \Delta(G)\leq 1/4$. Hence we need an upper bound on
$\rho_{a_1a_2}+\rho_{a_1b}$.

\smallskip
For that purpose we now specify $v_1,v_2$. The vertex $v$ is chosen as the
vertex of the maximum degree so that $e_1=\Delta$. We choose $v_2$ to have
maximum degree among all vertices in $N_G(v_1)$. The latter choice gives us
the estimate $\rho_{a_1a_2} + \rho_{a_1b}\leq 2\Delta e_2$ since
$\rho_{a_1a_2}+\rho_{a_1b}$ is simply the overall density of edges incident
to $A_1$. Putting all this together, we arrive at the estimate
$$
\beta(G) \leq f(\rho,\Delta,e_2) \df \frac{(1-2e_2)(4\Delta^2\rho-4\Delta^2e_2+
4\Delta\rho e_2-4\Delta e_2^2-4\Delta\rho +2\Delta e_2+\rho)}{8(1+2\Delta
e_2-\Delta -e_2)(1-\Delta-e_2)}.
$$
Let us also remind that we have the constraints
$$
0\leq \rho, e_2 \leq \Delta \leq 1/4.
$$

This optimization problem is a bit nasty to be fully analyzed, i.e. give an
analytical estimate on $\beta(G)$ in terms of $\rho(G)$. Instead, we compute
$$
\frac 1{50} - f(\rho,\Delta,e_2) = \frac{Q(\rho,\Delta,e_2)}{200(1-\Delta-e_2)
(1-\Delta-e_2+2\Delta(e_2))},
$$
where $Q$ is a polynomial. Our goal is to show that $\rho\leq \rho_0$ implies
$Q(\rho,\Delta,e_2)\geq 0$.

We note that $Q(\rho_0,\rho_0,\rho_0)=0$ and that individual degrees of $Q$
in $\rho,\Delta,e_2$ are 1, 2 and 3, respectively.

We first compute
$$
\frac{\partial Q}{\partial \rho} = -25(1-2e_2)(4\Delta^2+4\Delta e_2-4\Delta+1)
\leq  -25(1-2e_2)(1-2\Delta)^2 \leq 0.
$$
Hence it is sufficient to prove that
$$
Q_1(\Delta,e_2) \df Q(\min(\rho_0,\Delta),\Delta,e_2) \geq 0.
$$
$Q_1$ is no longer smooth in $\Delta$ but it is still a cubic polynomial in
$e_2$. We consider two cases: $e_2\leq\rho_0$ and $e_2\geq\rho_0$.

If $e_2\leq\rho_0$ then we consider Taylor's coefficients at $e_2=\rho_0$:
$$
\left.\frac 1{r!}\frac{\partial^r Q_1(\Delta,e_2)}{(\partial
e_2)^r}\right|_{e_2=\rho_0}\ (r=0..3),
$$
 and it turns out (see the Maple worksheet)
that they are non-negative for even $r$ and negative for odd $r$. The
required inequality $Q_1(\Delta,e_2)\geq 0$ follows.

In the second case $e_2\geq \rho_0$ we also have $\Delta\geq \rho_0$ and
hence $Q_1(\Delta,e_2)=Q(\rho_0,\Delta, e_2)$ is a quadratic polynomial in
$\Delta\in [e_2,1/4]$. It should be noted that it can be either convex or
concave. But in either case the required inequality $Q_1(\Delta,e_2)\geq 0$
follows from $Q_1(e_2,e_2)\geq 0$, $Q_1(1/4,e_2)\geq 0$ and $\frac{\partial
Q_1}{\partial \Delta}|_{\Delta=e_2}\geq 0$.

\subsection{Strongly regular case} \label{sec:strongly_regular}
In this section we prove Theorem \ref{thm:strongly_regular}. For a brief
background on triangle-free strongly regular (TFSR in what follows) graphs we
follow \cite{Big2}.

\smallskip
Except for the complete bipartite graphs $K_{n,n}$ (for which Erd\H{o}s's
conjecture is vacuously true), there are seven known examples of TFSR graphs;
the cycle $C_5$, the Petersen graph and the Clebsch graph being the smallest.
The obvious parameters of a TFSR graph $G$ are $n$, $k$ (the degree of a
vertex) and $c$ (the number of common neighbours of a pair of non-adjacent
vertices). They are actually related as
$$
n=1+\frac kc(k-1+c).
$$

From now on we assume that $G$ is different from $K_{n,n}$ {\em and} that it
is different from $C_5$. Then the quantity $s\df \sqrt{c^2+4(k-c)}$ is an
integer, and the only positive eigenvalue of the adjacency matrix different
from $k$ is given by $q\df \frac{s-c}2.$ It is also an integer such that
\begin{equation} \label{eq:bound_on_c}
1\leq c\leq q(q+1).
\end{equation}
Furthermore, we have
$$
k=(q+1)c+q^2,
$$
hence $k$ and $n$ are rational functions in $c$ and $q$. In particular, we
compute
$$
\rho(G) = \frac kn = \frac{c(qc+q^2+c)}{(qc+q^2+2c+q)(qc+q^2+ c-q)}\df Q(q,c).
$$
Let us now analyze this expression.

Firstly,
$$
\frac{\partial Q}{\partial c} = \frac{q(q+1)(c^2q^2+2cq^3+q^4+c^2q-q^3-c^2-2qc)}
{(qc+q^2+2c+q)^2(qc+q^2+c-q)^2} \geq 0,
$$
hence $Q$ is increasing in $c$.

Next, let
$$
Q_1(q) \df Q(q,q(q+1)) = \frac{q^2+3q+1}{q(q+3)^2}
$$
(this corresponds to so-called {\em Krein graphs}). Then
$$
Q_1(q)' = -\frac{q^3+3q^3+3q+3}{q^2(q+3)^3}<0,
$$
hence $Q_1(q)$ is decreasing. As $Q_1(4)=\frac{29}{196}<\rho_0$, the proof of
Theorem \ref{thm:strongly_regular} boils down to the three cases $q=1,2,3$:
all others are taken care of by Theorem \ref{thm:small_density}.

\medskip
When $q=1$, we have either the Petersen graph ($c=1$) or the Clebsch graph
($c=2$). Conjecture \ref{conj:erdos} for the Clebsch graph is verified by the
half $(N_G(u)\cup N_G(v)) \setminus \{u,v\}$, where $(u,v)$ is an arbitrary
edge.

\smallskip
When $q=2$, we have $Q(2,1)=\frac{7}{50}<\rho_0$ (this is the
Hoffman-Singleton graph) hence it is sufficient to consider the cases $2\leq
c\leq 6$. Well-known ``arithmetic conditions'' rule out $c\in \{3,5\}$
\cite[Table 1]{Big2}, and the three other cases correspond precisely to the
remaining known TFSR graphs: Gewirtz, M22 and Higman-Sims (they are unique
for their values of $c,q$ \cite{Gew,Brou,Gew2}).

For the Gewirtz graph, we pick up four vertices $v_1,v_2,v_3,v_4$ spanning an
induced matching with two edges and consider the half (see the Maple
worksheet) $\bigcup_{i=1}^4 N_G(u_i)\setminus \{u_1,u_2,u_3,u_4\}$. It spans
51 edges which proves $\beta(G)\leq 0.017$.

When $G$ is the $M_{22}$ graph, we similarly let $A\df \bigcup_{i=1}^3
N_G(u_i)\setminus \{u_1,u_2,u_3\}$, where $(u_1,u_2)\in E(G)$ and $u_3\not\in
N_G(u_1)\cup N_G(u_2)$. Then $|A|=38$ and $|E(A)|=109$. Moreover, there
exists a vertex $v\not\in A$ such that $|N_G(v)\cap A|=9$. Adding to $A$ half
of the vertex $v$, we get a half witnessing $\beta(G)\leq 0.0192$.

For the Higman-Sims graph we present an ad hoc half achieving $\beta(G)\leq
\frac 1{50}-10^{-4}$. It was found by a simple optimization program
remarkably suggesting that this bound is actually tight. If it is true (we
did not attempt to verify the claim with a rigorous argument) then the
Higman-Sims graph comes very close to the bound in Erd\H{o}s's conjecture.

\smallskip
Finally, when $q=3$ we have $Q(3,11)=\frac{583}{3350}<\rho_0$. Hence the only
case to consider is $q=3,c=12$ i.e. a hypothetical 57-regular Krein graph on
324 vertices. A simple solution is to note that such a graph is known not to
exist \cite{GaM,KaOs}. Let us, however, sketch another argument due to
Grzesik and Volec (unpublished) that in our opinion is more instructive and
may be of independent interest.

As we already noticed, in the bound \eqref{eq:mu0} the quantity $\rho_b$ can
be eliminated via the identity
$\rho=\rho_{a_1a_2}+\rho_{a_1b}+\rho_{a_2b}+\rho_b$. If $G$ is also known to
be regular, then $p_0=\frac{1/2-2\rho}{1-2\rho}$ and $\rho_{a_ib}$ can be
also eliminated using $\rho_{a_ib}+\rho_{a_1a_2}=2\rho^2$. Plugging all this
into \eqref{eq:mu0} and averaging over all choices of the edge $(v_1,v_2)$,
as in Section \ref{sec:absolute}, we arrive at the bound
\begin{equation} \label{eq:regular}
\beta(G) \leq \frac{\frac 23C_4+\rho^2(1-4\rho)}{8\rho(1-2\rho)^2}
\end{equation}
that holds for {\em any regular} (triangle-free) graph $G$.

Now, if $G$ is also strongly regular then $C_4(G)$ can be easily calculated
as
$$
C_4(G) = \frac 3{n^3} (k^2+c^2(n-k-1))
$$
(recall from Section \ref{sec:prel} that $C_4(G)$ counts degenerated cycles
as well!) Substituting this into \eqref{eq:regular}, we get
$$
\beta(G) \leq \frac{c(cq+q^2-c)}{8q(q+1)(c+q)(c+q-1)}.
$$
In particular, when $q=3, c=12$ we have $\beta(G)\leq \frac{11}{560}$.

\subsection{Graphs with large independence number}
\label{sec:large_independence}

In this section we prove Theorem \ref{thm:large_alpha}; as we noted in the
introduction, for $\alpha\geq 2/5$ it generalizes several previously known
results.

It will be convenient to assume that $n$ is even: this can be always achieved
by replacing each vertex with two identical twins. Let $A\subseteq V(G)$ be
an independent set with $p_a=\alpha\geq 3/8$. We build a larger set
$B\supseteq A$ by recursively adding to it vertices that bring with them only
a few edges. More exactly, apply the following simple algorithm:
\begin{figure}[h]
\begin{center}\fbox{\begin{minipage}{100cm}\begin{tabbing}
$B:=A$\\
{\bf while}\ $|B|<n/2$\ {\bf and}\ $\exists v\not\in B\of{e_B(v)\leq \frac
12-\alpha}$\\
{\bf do} $B:=B\cup \{v\}$.
\end{tabbing}\end{minipage}}\end{center}\end{figure}

If this algorithm terminates since $B$ reaches size $n/2$ then $\beta(G,B)
\leq \of{\frac 12-\alpha}^2$ which is $\leq \frac 12\alpha\of{\frac 12-\alpha}$
since $\alpha\geq \frac 38>\frac 13$ and we are done. Hence we can assume
w.l.o.g. that the algorithm stops when the required vertex $v$ no longer exists. Thus,
we now have a set $B$ such that:
\begin{equation} \label{eq:bounds_on_b}
\begin{cases}
p_b\in [\alpha,1/2]\\
\rho_b \leq 2\of{\frac 12-\alpha}(p_b-\alpha)\\
\forall v\not\in B \of{e_B(v)>\frac 12-\alpha}.
\end{cases}
\end{equation}

\smallskip
Let
$$
C\df \set{v\not\in B}{e_B(v)> \frac{p_b}{2}}.
$$
Then $C$ is independent (as any two vertices of $C$ have a common neighbor in
$B$). Hence $p_c\leq \alpha$ from the definition of $\alpha(G)$. This allows
us to choose a set of vertices $D$ disjoint from both $B$ and $C$ and such that
$p_d=1-\alpha-p_b$. We now consider two cases, depending on whether there
exists a vertex in $B$ that has many neighbors in $D$ or not.

\smallskip\noindent
{\bf Case 1. There exists $v\in B$ such that $e_D(v)\geq \frac 12-p_b$.}

Pick up arbitrarily $E\subseteq N_G(v)\cap D$ with $p_e=\frac 12-p_b$; note that
$E\subseteq N_G(v)$ is independent. Also, for every $v\in E$ we have $e_B(v)\leq
\frac{p_b}2$ since $E\subseteq D$. Then we have (note the absence of the
coefficient 2 in the last term!)
$$
2\beta(G, B\cup E) \leq 2\of{\frac 12 - \alpha}(p_b-\alpha) +p_b\of{\frac
12-p_b}.
$$
The right-hand side is a concave quadratic function in $p_b$, with maximum at
$p_b=\frac 34-\alpha$ which is $\leq\alpha$ since we assumed $\alpha\geq 3/8$.
Hence, since $p_b\geq\alpha$ we can plug in $p_b:=\alpha$ and this completes
the analysis of Case 1.

\smallskip\noindent
{\bf Case 2. For any $v\in B$ we have $e_D(v)\leq \frac 12-p_b$.}

This case is slightly more elaborate. Let us first fix an individual $v_0\in B$
(we will later average over this choice). Let $E\df N_G(v_0)\cap D$ (so that
$p_e\leq \frac 12-p_b$) and $F\df
D\setminus N_G(v_0)$; thus, $D= E\stackrel. \cup F$ with $E$ independent. Consider
the half
$$
\mu(v) \df \begin{cases} 1& \text{if}\ v\in B\cup E\\
p& \text{if}\ v\in F\\
0& \text{in all other cases,}
\end{cases}
$$
where
$$
p = \frac{1/2-p_b-p_e}{p_f}.
$$
Then
$$
2\beta(G,\mu) =\rho_b+\rho_{be}+p\rho_{bf}+p\rho_{ef}+p^2\rho_f.
$$

The bound on $\rho_b$ is given by \eqref{eq:bounds_on_b}, and we have $\rho_{bf}
\leq p_bp_f$; the coefficient 2 is absent for the same reasons as above.
For $\rho_{ef}$ we use the trivial bound $\rho_{ef}\leq 2p_ep_f$
and, finally $\rho_f\leq \frac{p_f^2}{2}$ simply because $G$ is triangle-free.
Plugging all this into the above bound (we leave $\rho_{be}$ alone for the time
being) we get
\begin{equation} \label{eq:bound_on_beta}
\begin{cases} 2\beta(G,\mu) & \leq 2\of{\frac 12-\alpha}(p_b-\alpha) + \rho_{be}
+ p_b\of{\frac 12-p_b-p_e}
\\ & \hspace{\parindent} +2p_e\of{\frac 12-p_b-p_e} + \frac 12\of{\frac 12 -p_b-p_e}^2 \\ & =
2\of{\frac 12-\alpha}(p_b-\alpha) + \frac 12 \of{\frac 12-p_b-p_e} \of{\frac
12+p_b+3p_e}\\ & \hspace{\parindent} + \rho_{be}.
\end{cases}
\end{equation}
In this bound, $p_e$ and $\rho_{be}$ are the only quantities that depend on the
choice of $v_0\in B$, and we now randomize over all such choices.

The bound \eqref{eq:bound_on_beta} is concave in $p_e$ hence we may simply
replace $p_e$ with its expected value $\frac{\rho_{bd}}{2p_b}$.

As for $\rho_{be}$, pick $\rn w\in_R D$ uniformly at random; then by a standard
double counting we see that
$$
\expect{\rho_{be}} = \frac{2p_d}{p_b} \expect{e_B(\rn w)^2}.
$$
But we also know that
$$
\frac 12-\alpha \leq e_B(\rn w) \leq \frac{p_b}{2},
$$
where the first inequality comes from \eqref{eq:bounds_on_b} while
the second follows from $D\cap C=\emptyset$. Moreover,
$$
\expect{e_B(\rn w)} = \frac{\rho_{bd}}{2p_d}.
$$
Estimating the second moment in a standard way, we get
$$
\expect{e_B(\rn w)^2} \leq \frac{\rho_{bd}}{2p_d}\of{\frac
12-\alpha+\frac{p_b}{2}} - \frac{p_b}2\of{\frac 12-\alpha}.
$$
Finally, plugging all our findings into \eqref{eq:bounds_on_b}, we get
\begin{eqnarray*}
\beta &\leq& Q(\alpha,p_b,\rho_{bd})\\ &\df&  \frac{8\alpha^2p_b^2-24\alpha p_b^3 - 4p_b^4+
4\alpha p_b^2 - 8\alpha p_b\rho_{bd}+12p_b^3-4p_b^2 \rho_{bd} -3p_b^2
+6p_b\rho_{bd} -3\rho_{bd}^2}{16p_b^2}
\end{eqnarray*}
(see the Maple worksheet).

$Q$ is quadratic concave in $\rho_{bd}$ and, as before, $\rho_{bd}\leq
p_bp_d= p_b(1-\alpha-p_b)$ since $D\cap C=\emptyset$. Moreover,
$$
\left.\frac{\partial Q}{\partial \rho_{bd}}\right|_{\rho_{bd}=p_bp_d} =
\frac{p_b-\alpha}{8p_b}\geq 0.
$$
Hence
$$
Q(\alpha,p_b,\rho_{bd}) \leq Q(\alpha,p_b,p_b(1-\alpha-p_b)) =
\frac{13}{16}\alpha^2 - \frac 98\alpha p_b-\frac 3{16}p_b^2 -
\frac 14\alpha +\frac 12p_b \df Q_1(\alpha,p_b).
$$
Finally, $Q_1$ is quadratic concave in $p_b$ and $\left.\frac{\partial Q_1}{\partial
p_b}\right|_{p_b=\alpha} = \frac{1-3\alpha}2<0$ (as $\alpha\geq \frac 38$). Since $p_b\geq \alpha$,
we get $Q_1(\alpha,p_b)\leq Q_1(\alpha,\alpha) =\frac{\alpha}{2}\of{\frac 12-\alpha}$.
This completes the proof.

\subsection{Graphs of girth $\geq 5$}

In this section we prove Theorem \ref{thm:girth}, and for this particular proof
we resort to absolute sizes of the sets involved rather than their densities.
The reason is that the girth assumption does not survive blowing up a graph, and
this makes the density-based language unnatural.

So we fix a triangle-free graph $G$ with $g(G)\geq 5,\ |V(G)|=n$, and let
$v_0\in V(G)$ be a vertex of the maximum degree $k$. We may assume that $k\leq
\frac{n-1}{2}$ (otherwise the result is trivial) and also that $G$ is a {\em
minimal} counterexample to Erd\H{o}s's conjecture, that is $\beta(G^\ast)\leq
\frac 1{50}$ for any proper induced subgraph $G^\ast$ of $G$. We let
$$
A\df N_G(v_0),\ \ B\df V(G)\setminus (\{v_0\}\cup A).
$$
Then $g(G)\geq 5$ implies
\begin{equation} \label{eq:edges_ab}
\forall v\in B (|N_G(v)\cap A|\leq 1).
\end{equation}

We now apply the minimality assumption to the induced subgraph $G|_B$. This
gives us a function $\nu\function B{[0,1]}$ such that
\begin{equation} \label{eq:nu_bound}
\begin{cases}
\sum_{v\in B} \nu(v) = \frac{n-k-1}2\\
\sum_{(u,v)\in E(B)} \nu(u)\nu(v) \leq \frac{(n-k-1)^2}{50}.
\end{cases}
\end{equation}
We use it to define a half $\mu$ in the whole graph $G$ as follows:
$$
\mu(v) \df \begin{cases}0,& v=v_0\\ 1,& v\in A\\ p\nu(v),& v\in B,\end{cases}
$$
where
$$
p\df \frac{n-2k}{n-k-1}.
$$
Then we have
\begin{equation} \label{eq:two_ways}
\beta(G,\mu) \leq \frac 1{n^2}\of{\frac{(n-2k)^2}{50} + \sum_{(u,v)\in E(A,B)}\mu(u)\mu(v)}.
\end{equation}
We will employ two different methods of bounding the term $\sum_{(u,v)\in E(A,B)}\mu(u)\mu(v)$.

Firstly, by \eqref{eq:edges_ab} we have
$$
\sum_{(u,v)\in E(A,B)}\mu(u)\mu(v) \leq \sum_{v\in B} \mu(v)=\frac n2-k
$$
and thus
\begin{equation} \label{eq:first_bound}
\begin{cases}\beta(G,\mu) &  \leq \frac 1{n^2} \of{\frac{(n-2k)^2}{50} + \frac n2-k} \\ & = \frac
1{50} - \frac 1{5n^2} \of{n(4k-25) +50k-4k^2}.\end{cases}
\end{equation}
We now start a case analysis.

\smallskip\noindent
{\bf Case 1. $k\geq 7$.}

In this case, since $n\geq 2k+1$, we have $n(4k-25) +50k-4k^2\geq (2k+1)(4k-25)
+50k-4k^2 = 4k^2+4k-25\geq 0$, and we are done by \eqref{eq:first_bound}.

\smallskip\noindent
{\bf Case 2. $k\leq 6$.}

This time, the same condition $n(4k-25) +50k-4k^2<0$ (that can be assumed w.l.o.g.)
provides a new {\em lower} bound on $n$
\begin{equation} \label{eq:lower_on_n}
n\geq \left\lceil {\frac{2k(25-2k)}{25-4k}} \right\rceil.
\end{equation}

To get an upper bound on $n$, we estimate the term $\sum_{(u,v)\in E(A,B)}\mu(u)\mu(v)$
as $(k-1)k$, simply because the degree of any vertex in $A$ is $\leq k$, and all of
them are adjacent to $v_0\not\in B$. Thus
$$
\beta(G,\mu) \leq \frac 1{n^2} \of{\frac{(n-2k)^2}{50} + (k-1)k} =\frac 1{50} -
\frac k{25n^2}(2n+25-27k).
$$
Hence we can also assume that
\begin{equation} \label{eq:upper_on_n}
n \leq \left\lceil \frac{27}2(k-1) \right\rceil
\end{equation}
which immediately rules out the case $k=1$. Also, \eqref{eq:lower_on_n} and
\eqref{eq:upper_on_n} rule out the case $k=6$ as well which leaves us with
the possibilities $k=2,3,4,5$ and 80 potential values for the pair $(k,n)$.

Instead of trying to do the remaining analysis manually, we employ a different
strategy. Namely, we record our argument in the form of ``unprocessed'' (and
recursive) bounds, without attempting to simplify them, and then we simply feed
the formulas to Maple to finish the job.

\smallskip
To start with, let $C\df V(G)\setminus (A\cup N_G(A))$ be the set of vertices
at distance $\geq 3$ from $v_0$; note that
$$
|C| \geq n-k^2-1.
$$
Let $R(3,u)$ be the off-diagonal Ramsey number; we will only use the following
well-known small values:
$$
R(3,0)=0,\ R(3,1)=1,\ R(3,2)=3,\ R(3,3)=6,\ R(3,4) = 9,\ R(3,5)=14.
$$
For every $u\in \left[ 0, \left\lceil n/2-k \right\rceil \right]$ such that $R(3,u) \leq n-k-1\ (=
|B|)$ we are going to derive its own bound $\beta(G)\leq \beta_u$, and then we
will minimize over all choices of $u$. So let us fix for the time being
some $u$ with the above properties.

\smallskip
Pick a subset $B_u\in {B\choose R(3,u)}$ with the only restriction that it contains
as many vertices in $C$ as possible. Then we have
\begin{equation} \label{eq:e_bound}
|E(A, B_u)| = |B_u\setminus C| = R(3,u) \dotdiv |C| \leq R(3,u) \dotdiv (n-k^2-1),
\end{equation}
where $x\dotdiv y\df \max(0,x-y)$.
Finally, let $B_u'\subseteq B_u$ be an independent subset of size $u$ existing
by the definition of Ramsey numbers. Further analysis splits into two more cases.

\smallskip\noindent
{\bf Case 2.1. $u=\left\lceil n/2-k\right\rceil$.}

If $n$ is even, we take the half $A\cup B_u'$. If $n$ is odd, we can assume w.l.o.g.
that $B_u'\cap N_G(A)\neq\emptyset$ (as otherwise we are done). Let $\mu$ be the half
obtained from $A\cup B_u'$ by removing half a vertex in $B_u'\cap N_G(A)$; this will
give us an extra saving of half-edge.

We have two different estimates on $|E(A, B_u')|$: one follows from \eqref{eq:e_bound}
and, on the other hand we, like before, have the trivial bound $|E(A,B_u')| \leq
n/2-k\leq u$ coming from \eqref{eq:edges_ab}. Summarizing,
\begin{equation} \label{eq:beta_u_case1}
\begin{cases}\beta(G) &\leq \beta_u\\ & \df \frac 1{n^2}\of{ \min(u, R(3,u)-
(n-k^2-1)) \dotdiv \frac 12 (n\bmod 2)}\\ & (u = \left\lceil n/2-k \right\rceil).\end{cases}
\end{equation}
Let us stress that this bound is defined only when $R\of{3,\left\lceil n/2-k \right\rceil}
\leq n-k-1$.

\smallskip\noindent
{\bf Case 2.2. $u< \left\lceil n/2-k\right\rceil$.}

In this case we have\footnote{We do not need the half-edge saving from the
previous case.}
\begin{equation} \label{eq:beta_u_case2}
\beta(G) \leq \beta_u \df \gamma(k,n,k+u, \min(u, R(3,u) \dotdiv (n-k^2-1))),
\end{equation}
where the function $\gamma(k,n,t,e)\ (t\leq n/2)$ abstracts our situation as
follows:
$$
\gamma(k,n,t,e) \df \max_{G,A} \min_{\mu|_A\equiv 1} \beta(G,\mu)\ (t\leq \left\lfloor
n/2 \right\rfloor).
$$
Here $G$ runs over all graphs with $n$ vertices and $\Delta(G)\leq k$, $A$
runs over all sets of vertices with $|A|=t$ and $\absvalue{E(A)} \leq e$
and $\mu$ runs over all halves containing $A$.
What remains is to give sufficiently good (for our purposes) recursive bounds on
$\gamma$.

First of all, when $n$ is even and $t=n/2$, we clearly have
\begin{equation} \label{eq:gamma_even}
\gamma(k,n,n/2,e) = \frac e{n^2}\ (n\ \text{is even}).
\end{equation}

Next, assume that $n$ is odd and $t=\frac{n-1}2$. Fix the worst-case $G,A$,
and let $e^\ast\leq e$ be the actual number of edges in $G|_A$. Then $|E(A,
V(G)\setminus A)|$ has at most $kt-2e^\ast$ edges. Hence there exists a
vertex $v\not\in A$ with
\begin{equation} \label{eq:good_v}
|N_G(V)\cap A| \leq \left\lfloor
\frac{kt-2e^\ast}{n-t}\right\rfloor.
\end{equation}
Adding to $A$ half of that vertex, we conclude
\begin{equation} \label{eq:gamma_odd}
\gamma(k,n,t,e) \leq \max_{0\leq e^\ast\leq e} \frac 1{n^2} \of{e^\ast
+\frac 12 \left\lfloor \frac{kt-2e^\ast}{n-t}\right\rfloor}\ (n\ \text{odd}, t =
\frac{n-1}2).
\end{equation}

Similarly, for smaller values of $t$ we apply recursion by
letting $A:= A \cup \{v\}$, where $v$ is he vertex satisfying
\eqref{eq:good_v}. This gives us
\begin{equation} \label{eq:gamma_recursive}
\gamma(k,n,t,e) \leq \gamma\of{k,n,t+1,\max_{0\leq e^\ast\leq e}\of{e^\ast+
\left\lfloor \frac{kt-2e^\ast}{n-t}\right\rfloor}}\ (t< \left\lfloor n/2\right\rfloor).
\end{equation}

This completes our description of $\beta_u$ in the case 2.2.

\smallskip
Finally, the ``master formula'' now reads as
\begin{equation} \label{eq:beta_u_min}
\beta(G) \leq \min\set{\beta_u}{0\leq u\leq \left\lceil n/2-k\right\rceil \land
R(3,u)\leq n-k-1}.
\end{equation}
The bounds \eqref{eq:beta_u_min}, \eqref{eq:beta_u_case1},
\eqref{eq:beta_u_case2}, along with recursive estimates \eqref{eq:gamma_even},
\eqref{eq:gamma_odd}, \eqref{eq:gamma_recursive} on the auxiliary function $\gamma$
suffice to complete the analysis of the 80 remaining cases. See the Maple
worksheet for details.

\section{Conclusion} \label{sec:concl}

In this paper we have proved several partial results on Erd\H{o}s's half-graph
conjecture. While they make this conjecture even more plausible, it still
remains wide open. The same is true for the last of Erd\H{o}s's conjectures
on this subject: prove that any triangle-free graph on $n$ vertices can be made
bi-partite by removing at most $\frac{n^2}{25}$ edges.

As for intermediate, and probably more accessible, goals we would like to ask to
extend Theorem \ref{thm:large_alpha} to a neighbourhood of the critical value
$\alpha = 2/5$, i.e. prove the half-graph conjecture for triangle-free graphs $G$
with $\alpha(G)\geq 2/5-\epsilon$ for a fixed constant $\epsilon>0$. As we noted
above, such an improvement if known for the minimum degree and the average degree
\cite{Kri,KeS2}.

We have highlighted the extremal problem of finding the minimal density
of quadriliterals in triangle-free graphs with given edge density and have given
its applications to the sparse half problem. Since this quantity can be viewed
(actually, in a quite precise sense) as the measure of non-randomness in a graph,
perhaps it might be worth studying in its own right.

\section*{Acknowledgment}
I would like to thank Andrzej Grzesik and Jan Volec for pointing out the
references \cite{GaM,KaOs} and for sharing with me the alternate argument
sketched at the end of Section \ref{sec:strongly_regular}.

\newcommand{\etalchar}[1]{$^{#1}$}

\end{document}

%% file: razb.tex
\newlength{\longeqmarginwidth}
\newlength{\longeqwidth}
\newlength{\longeqskiplength}
\newcommand{\Vzero}[2]{\mathord{\stackrel{\textstyle\kern1pt\circ}
{\smash V\vbox to6pt{}}}\vphantom{V}_{#1}^{#2}}

\newcommand{\function}[2]{:#1 \longrightarrow #2}
\newcommand{\of}[1]{\left( #1 \right)}
\newcommand{\absvalue}[1]{\left| #1 \right|}
\newcommand{\set}[2]{\left\{\hspace{0.2ex} #1 \left|\: #2
\right. \right\}}

\newcommand{\vecset}[2]{\left(\hspace{0.2ex} #1 \left|\: #2 \right. \right)}

\newcommand{\eval}[2]{\llbracket #1 \rrbracket_{#2}}

\newcommand{\df}{\stackrel{\rm def}{=}}

\newcommand{\rn}{\boldsymbol}

\newcommand{\expect}[1]{ {\bf E}\! \left[ #1 \right] }

\newcounter{operator}

%% file: types.eepic
\setlength{\unitlength}{0.254mm}
\begin{picture}(372,222)(30,-356)
        \special{color rgb 0 0 0}\allinethickness{0.254mm}\special{sh 0.99}\put(40,-195){\ellipse{4}{4}} 
        \special{color rgb 0 0 0}\allinethickness{0.254mm}\special{sh 0.99}\put(90,-195){\ellipse{4}{4}} 
        \special{color rgb 0 0 0}\put(60,-211){\shortstack{$E$}} 
        \special{color rgb 0 0 0}\allinethickness{0.254mm}\special{sh 0.99}\put(215,-145){\ellipse{4}{4}} 
        \special{color rgb 0 0 0}\allinethickness{0.254mm}\special{sh 0.99}\put(215,-195){\ellipse{4}{4}} 
        \special{color rgb 0 0 0}\allinethickness{0.254mm}\special{sh 0.99}\put(165,-145){\ellipse{4}{4}} 
        \special{color rgb 0 0 0}\allinethickness{0.254mm}\special{sh 0.99}\put(165,-195){\ellipse{4}{4}} 
        \special{color rgb 0 0 0}\allinethickness{0.254mm}\special{sh 0.99}\put(215,-145){\ellipse{4}{4}} 
        \special{color rgb 0 0 0}\allinethickness{0.254mm}\special{sh 0.99}\put(215,-195){\ellipse{4}{4}} 
        \special{color rgb 0 0 0}\put(185,-211){\shortstack{$\sigma_1$}} 
        \special{color rgb 0 0 0}\allinethickness{0.254mm}\special{sh 0.99}\put(290,-145){\ellipse{4}{4}} 
        \special{color rgb 0 0 0}\allinethickness{0.254mm}\special{sh 0.99}\put(290,-195){\ellipse{4}{4}} 
        \special{color rgb 0 0 0}\allinethickness{0.254mm}\special{sh 0.99}\put(340,-145){\ellipse{4}{4}} 
        \special{color rgb 0 0 0}\allinethickness{0.254mm}\special{sh 0.99}\put(340,-195){\ellipse{4}{4}} 
        \special{color rgb 0 0 0}\put(310,-211){\shortstack{$\sigma_2$}} 
        \special{color rgb 0 0 0}\allinethickness{0.254mm}\special{sh 0.99}\put(105,-290){\ellipse{4}{4}} 
        \special{color rgb 0 0 0}\allinethickness{0.254mm}\special{sh 0.99}\put(105,-340){\ellipse{4}{4}} 
        \special{color rgb 0 0 0}\allinethickness{0.254mm}\special{sh 0.99}\put(155,-290){\ellipse{4}{4}} 
        \special{color rgb 0 0 0}\allinethickness{0.254mm}\special{sh 0.99}\put(155,-340){\ellipse{4}{4}} 
        \special{color rgb 0 0 0}\put(125,-356){\shortstack{$\sigma_3$}} 
        \special{color rgb 0 0 0}\allinethickness{0.254mm}\special{sh 0.99}\put(280,-290){\ellipse{4}{4}} 
        \special{color rgb 0 0 0}\allinethickness{0.254mm}\special{sh 0.99}\put(280,-340){\ellipse{4}{4}} 
        \special{color rgb 0 0 0}\allinethickness{0.254mm}\special{sh 0.99}\put(230,-290){\ellipse{4}{4}} 
        \special{color rgb 0 0 0}\allinethickness{0.254mm}\special{sh 0.99}\put(230,-340){\ellipse{4}{4}} 
        \special{color rgb 0 0 0}\allinethickness{0.254mm}\special{sh 0.99}\put(280,-290){\ellipse{4}{4}} 
        \special{color rgb 0 0 0}\allinethickness{0.254mm}\special{sh 0.99}\put(280,-340){\ellipse{4}{4}} 
        \special{color rgb 0 0 0}\put(250,-356){\shortstack{$\sigma_4$}} 
        \special{color rgb 0 0 0}\allinethickness{0.254mm}\path(105,-340)(105,-290) 
        \special{color rgb 0 0 0}\allinethickness{0.254mm}\path(105,-340)(155,-340) 
        \special{color rgb 0 0 0}\allinethickness{0.254mm}\path(280,-340)(230,-340) 
        \special{color rgb 0 0 0}\allinethickness{0.254mm}\path(290,-145)(290,-195) 
        \special{color rgb 0 0 0}\allinethickness{0.254mm}\path(290,-195)(340,-195) 
        \special{color rgb 0 0 0}\put(155,-196){\shortstack{\scriptsize 1}} 
        \special{color rgb 0 0 0}\put(220,-196){\shortstack{\scriptsize 2}} 
        \special{color rgb 0 0 0}\put(155,-146){\shortstack{\scriptsize 3}} 
        \special{color rgb 0 0 0}\put(220,-146){\shortstack{\scriptsize 4}} 
        \special{color rgb 0 0 0}\allinethickness{0.254mm}\special{sh 0.99}\put(215,-145){\ellipse{4}{4}} 
        \special{color rgb 0 0 0}\allinethickness{0.254mm}\special{sh 0.99}\put(215,-195){\ellipse{4}{4}} 
        \special{color rgb 0 0 0}\allinethickness{0.254mm}\special{sh 0.99}\put(165,-145){\ellipse{4}{4}} 
        \special{color rgb 0 0 0}\allinethickness{0.254mm}\special{sh 0.99}\put(165,-195){\ellipse{4}{4}} 
        \special{color rgb 0 0 0}\allinethickness{0.254mm}\special{sh 0.99}\put(215,-145){\ellipse{4}{4}} 
        \special{color rgb 0 0 0}\allinethickness{0.254mm}\special{sh 0.99}\put(215,-195){\ellipse{4}{4}} 
        \special{color rgb 0 0 0}\put(155,-196){\shortstack{\scriptsize 1}} 
        \special{color rgb 0 0 0}\put(220,-196){\shortstack{\scriptsize 2}} 
        \special{color rgb 0 0 0}\put(155,-146){\shortstack{\scriptsize 3}} 
        \special{color rgb 0 0 0}\put(220,-146){\shortstack{\scriptsize 4}} 
        \special{color rgb 0 0 0}\put(30,-196){\shortstack{\scriptsize 1}} 
        \special{color rgb 0 0 0}\put(95,-196){\shortstack{\scriptsize 2}} 
        \special{color rgb 0 0 0}\put(280,-196){\shortstack{\scriptsize 1}} 
        \special{color rgb 0 0 0}\put(345,-196){\shortstack{\scriptsize 2}} 
        \special{color rgb 0 0 0}\put(280,-146){\shortstack{\scriptsize 3}} 
        \special{color rgb 0 0 0}\put(345,-146){\shortstack{\scriptsize 4}} 
        \special{color rgb 0 0 0}\put(220,-341){\shortstack{\scriptsize 1}} 
        \special{color rgb 0 0 0}\put(285,-341){\shortstack{\scriptsize 2}} 
        \special{color rgb 0 0 0}\put(220,-291){\shortstack{\scriptsize 3}} 
        \special{color rgb 0 0 0}\put(285,-291){\shortstack{\scriptsize 4}} 
        \special{color rgb 0 0 0}\put(95,-341){\shortstack{\scriptsize 1}} 
        \special{color rgb 0 0 0}\put(160,-341){\shortstack{\scriptsize 2}} 
        \special{color rgb 0 0 0}\put(95,-291){\shortstack{\scriptsize 3}} 
        \special{color rgb 0 0 0}\put(160,-291){\shortstack{\scriptsize 4}} 
        \special{color rgb 0 0 0}\allinethickness{0.254mm}\path(165,-195)(215,-195) 
        \special{color rgb 0 0 0}\allinethickness{0.254mm}\path(105,-340)(155,-290) 
        \special{color rgb 0 0 0}\allinethickness{0.254mm}\path(230,-290)(280,-290) 
        \special{color rgb 0 0 0}\allinethickness{0.254mm}\path(40,-195)(90,-195) 
        \special{color rgb 0 0 0}\allinethickness{0.254mm}\path(230,-290)(280,-340) 
        \special{color rgb 0 0 0}\allinethickness{0.254mm}\path(280,-290)(230,-340) 
        \special{color rgb 0 0 0} 
\end{picture}